\documentclass[12pt,a4paper]{amsart}
\usepackage{amsmath, amssymb, amsthm, amscd, amsfonts,graphicx}
\usepackage[matrix,arrow]{xy}
 \newtheorem{theorem}{Theorem}[section]

\newtheorem{cor}[theorem]{Corollary}
\theoremstyle{definition}
\newtheorem{definition}[theorem]{Definition}
\theoremstyle{definition}
\newtheorem{remark}[theorem]{Remark}
\theoremstyle{definition}

\theoremstyle{definition}

\newtheorem{proposition}[theorem]{Proposition}

\begin{document}

\title[Definable Continuous Induction on Ordered Abelian
Groups]{Definable Continuous Induction on Ordered Abelian
Groups}

 \author[ Jafar S. Eivazloo]{Jafar S. Eivazloo$^1$ }

 \address{ Department of Mathematics, University of Tabriz, P.O.Box: 51666-14766, Tabriz,
Iran.}

 \email{eivazloo@tabrizu.ac.ir}



 \begin{abstract}
As mathematical induction is applied to prove statements on natural numbers, {\it continuous induction} (or, {\it real induction}) is a tool to prove some statements in real analysis.(Although, this comparison is somehow an overstatement.) Here, we first consider it on densely ordered abelian groups to prove Heine-Borel theorem (every closed and bounded interval is compact with respect to order topology) in those structures. Then, using the recently introduced notion of pseudo finite sets, we introduce  a first order definable version of {\it continuous induction} in the language of ordered groups and we use it to prove a
definable version of  Heine-Borel theorem on densely ordered abelian groups.
 \end{abstract}

\maketitle
\section{Introduction and Preliminaries}

As mathematical induction is a good proof technique for statements on natural numbers, some efforts have been made to provide a similar technique in the real field, see  \cite{Ch, H, XX1, XX2, K, Cl}. Here, we define {\it continuous induction} on densely ordered abelian groups. Our motivation is to give a definable version of {\it continuous induction} in ordered structures such as the definable version of mathematical induction in Peano arithmetic.

 A structure $(G,+,<)$ is an ordered abelian group if
$(G,+)$ is an abelian group, $(G,<)$ is a linear ordered set, and
for all $x,y,z$ in $G$, if $x<y$ then $x+z<y+z$. An ordered
abelian group is dense if for all $x<y$ in $G$ there exists $z\in
G$ such that $x<z<y$. If $G$ is not dense, then it is discrete,
i.e. it contains a smallest positive member. Throughout the paper, let $G$ be a fixed arbitrary  densely ordered abelian group. Intervals in $G$ are defined in common way like in the real field $\Bbb{R}$. Then, $G$ is equipped with the interval topology for which  open intervals in $G$ constitute a topology base.

A nonempty
subset $C$ of  $G$ is a Dedekind cut  if it is downward closed. A cut $C$ is proper if $C\neq
G$. A proper cut $C$ is said to be  a gap in $G$ if it does not
have a least upper bound in $G$. An ordered abelian group $G$
which has no any gaps is said to be  Dedekind complete. Clearly,
$G$ is Dedekind complete if and only if every nonempty and bounded
from above subset of $G$ has a least upper bound in $G$.

 In the following,  we formulate the continuous induction over an
ordered abelian group.
\begin{definition}
\noindent { \it  Continuous Induction} ({${\rm CI}$}): Let $S\subseteq G$ satisfy the following conditions: \\
$i$) there exists at least one element
$g\in G$ such that $(-\infty , g)\subseteq S$, \\
$ii$) for
every element $x \in G$, if  $(-\infty , x)\subseteq S$
then there exists  $y>x$ in $G$ such that $(-\infty , y)\subseteq S$.

Then, $S=G$.
\end{definition}
We refer the conditions $i$ and $ii$ above as the assumptions of ${\rm CI}$.
\section{The Strength of Continuous Induction}
In order to see the usefulness of proving by continuous induction, we refer to \cite{Ch}, \cite{Cl}, and \cite{K}. Here, we use ${\rm CI}$ to prove that Heine-Borel theorem holds in densely ordered abelian groups. Our proof in somehow is different from ones in the mentioned references. It is such that to be adaptable for the definable version in the next section.

 A linearly ordered set which is dense (in itself) and
has no end points is Dedekid complete if and only if it satisfies
the CI \cite{H}. A similar statement is true for ordered (dense
or non-dense) abelian groups.

\begin{proposition} \label{Ded-CI}
An ordered abelian group $G$ is Dedekind complete if and only if
it satisfies the ${\rm CI}$.
\end{proposition}
\begin{proof}
 Let $G$ be a  Dedekind complete ordered abelian
group, and $S\subset G$ satisfy the assumptions $i$ and $ii$ of ${\rm CI}$. Let $C=\{g
\in G| \ (-\infty , g)\subseteq  S\}$. Then, $C$ is a Dedekind
cut in $G$. It follows from the condition $ii$  of the ${\rm CI}$ that
$C$ does not have a least upper bound in $G$, and so by Dedekind
completeness of $G$, $C=G$. From that we have $S=G$. Conversely, if $G$
is not Dedekind complete then it does have a gap, say $C$. The
subset $C$  satisfies the  conditions  of ${\rm CI}$, but $C\neq G$. So
$G$ does not satisfy the ${\rm CI}$.
\end{proof}

Therefore, every densely ordered abelian group $G$ which satisfies
the ${\rm CI}$ is isomorphic to the ordered additive group of reals.
For some practical purposes, we also define a bounded version of ${\rm CI}$, {\it bounded continuous induction} (${\rm BCI}$), on ordered abelian groups.

\begin{definition}
For every $a<b$ in $G$ and every $S\subseteq [a,b)$, if \\
i) there exists $x>a$ such that
$[a,x)\subseteq S$, and \\
ii) for every $x>a$ for which $[a,x)\subseteq S$, there exists $y>x$ such that $[a,y)\subseteq S$, \\
then $S=[a,b)$.
\end{definition}

\begin{proposition} \label{CI=BCI}
On densely ordered abelian groups, ${\rm CI}$ is equivalent to ${\rm BCI}$.
\end{proposition}
\begin{proof}
 Let $G$ be a densely ordered abelian group
that satisfies ${\rm CI}$. Suppose that  $S\subseteq [a,b)$, for $a<b$
in $G$, satisfies the two conditions  of  ${\rm BCI}$. If $S$ is not the
whole of $[a,b)$, then $C=\{x\in G|\exists y \in S: \  x<y\}$ is a gap
in $G$ which is Dedekind complete by Proposition \ref{Ded-CI}. This is a
contradiction.

Conversely, let $G$ satisfy ${\rm BCI}$, and
the two conditions  of  ${\rm CI}$ hold for $S\subset G$. Then, there is an element $x\in G$ such
that $(-\infty , x)\subset S$. Let    $a<x$, and
$b>x$ be arbitrary. Then,  the subset $S'=[a,b)\cap S$ of $G$ satisfies the conditions  of
${\rm BCI}$. Thus, $S'=[a,b)$ and so $[a,b)\subset S$. Therefore, $S=G$.
\end{proof}

In the following theorem, we  give a proof for Heine-Borel theorem on  densely ordered abelian
groups by using   ${\rm CI}$. We also show that any densely ordered abelian
group which satisfies Heine-Borel theorem, holds   ${\rm CI}$ too.

{\it Heine-Borel } ($\rm HB$) theorem: Every closed bounded interval $[a,b]$ is compact.\footnote{Here,
we consider the order topology, i.e. a topology upon open
intervals $(a,b)$ in $G$.}

\begin{theorem} \label{HB-CI}
Any densely ordered abelian group satisfies ${\rm CI}$ if and only if it
 satisfies Heine-Borel theorem.
\end{theorem}
\begin{proof}
 Let $G$ be a densely ordered abelian group that
satisfies ${\rm CI}$. Take arbitrary elements $a<b$ in $G$. We show that
the closed and bounded interval $[a,b]$ is compact. Let
$\mathcal{U}$ $=\{U_i \}_{i\in I}$ be an open cover for $[a,b]$,
and let
$$S= \{ x \in [a,b) \ |\  [a,x] \ \rm is \ covered \ by \ a
\ finite \ subset \ of \ \mathcal{U} \}.$$

By Proposition \ref{CI=BCI}, $G$ satisfies ${\rm BCI}$. In the following claim we
show that the conditions $i$ and $ii$  of ${\rm BCI}$ on
$[a,b)$ hold for $S$, so we have $S=[a,b)$. On the other hand, there exists
$i_n \in I$ such that $b\in U_{i_n}$, and so $(b-\epsilon, \
b]\subseteq U_{i_n}$ for some $\epsilon \in G^{>0}$.  Since
$b-\epsilon \in S$, then there exists a finite subset of
$\mathcal{U}$ which covers $[a,\ b-\epsilon ]$. Consequently, $[a,b]$ is covered
by a finite subset of $\mathcal{U}$.

{\it \bf Claim.} The subset $S$ satisfies the conditions  of ${\rm BCI}$
on $[a,b)$.

{\it \bf Proof of the claim.}
 $i$) Since $\mathcal{U}$ covers $[a,b]$, then there exists $i\in I$ such that
$a\in U_i$, so $[a,x_0)\subseteq U_i$ for some $x_0 \in (a,b)$.
Hence, by the definition of $S$, $[a,x_0)\subseteq S$.

$ii$) Let $a<x<b$, and $[a,x)\subseteq S$. There is an $i_0\in I$ such that
$x\in U_{i_0}$, and so $(x-\epsilon, \ x+\epsilon)\subseteq
U_{i_0}$ for some $\epsilon \in G^{>0}$. Since $x - \epsilon \in
S$, then there is a finite subcover of $\mathcal{U}$ for $[a, \
x-\epsilon]$, say $\{U_{i_1}, U_{i_1},...,U_{i_n}\}$. So,
$\{U_{i_0}, U_{i_1}, U_{i_1},...,U_{i_n}\}$ is a finite subcover
for $[a, \ x+\delta]$ for every $0<\delta<\epsilon$. Thus,  $[a,\
x+\epsilon )\subseteq S$. \hfill End of the proof of the claim.

Conversely, suppose  $G$ satisfies Heine-Borel theorem, i.e.
all closed bounded intervals in $G$ are compact. If $G$ does not
satisfy ${\rm CI}$, then by proposition \ref{Ded-CI}, $G$ has a Dedekind gap,
say $C$. Take $a$ in $C$ and b in $G \setminus C$. Then,
$\{(-\infty ,\ x)\ | \  x\in C \ and \  x>a \} \cup \{ G\setminus
C \}$ is an open cover for $[a,b]$ which does not have any finite
subcover.
\end{proof}
 In the real field $\Bbb{R}$, a consequence of ${\rm HB}$ theorem is that any continuous function on a closed bounded interval is uniformly continuous. In a similar method, one can easily conclude  the following result in ordered abelian groups.

\begin{cor}
Let $G$ be an ordered  abelian group that satisfies ${\rm CI}$. Then, if
 $f:[a,b]\rightarrow G$ is a continuous function, then it is uniformly continuous on $[a,b]$.
\end{cor}

\section{Definable Continuous Induction}
\label{lincomp}

Let $\mathcal{L}=\langle+,<,0, \dots \rangle$ be an expansion of the  language of ordered
abelian  groups.
 A subset $D$ of an ordered abelian
group $\langle G,+,<,0\rangle$ is said to be $\mathcal{L}$-definable (with parameters) if  there exists an $\mathcal{L}$-formula $\varphi (x,\bar{g})$, with some parameters  $\bar{g}$   from $G$, such that $D$ consists of all elements  $a\in G$ for which the  $\mathcal{L}$-structure $G$ satisfies $\varphi (a,\bar{g})$. $D= \{a\in G| G\vDash \varphi (a,\bar{g})\}$.  Now, we can consider the
first order version of the continuous induction in $\mathcal{L}$ by
replacing the subset $S$ in the formulation  of ${\rm CI}$ with
a definable subset $D$. Let ${\rm DCI}$ denote the first order definable
version of ${\rm CI}$. Indeed,  ${\rm DCI}$  is  a schema in the language $\mathcal{L}$.

\begin{definition}
For
any $\mathcal{L}$-formula $\varphi(v,\bar{w})$, let ${\rm DCI}_{\varphi}$ denote  the $\mathcal{L}$-formula \\
$\forall \bar{w}((\exists s \forall v<s \varphi (v,\bar{w})\wedge
\forall v (\forall s<v \varphi (s,\bar{w})\rightarrow
 \exists u>v
\forall s<u \varphi (s,\bar{w})))
\rightarrow \forall v \varphi
(v,\bar{w}))$.

Now,  ${\rm DCI}$ is the scheme $\{{\rm DCI}_{\varphi}| \
\varphi(v,\bar{w}) \ \rm is \ an \ \mathcal{L}-formula   \}$.
\end{definition}
So, ${\rm DCI}$ is preserved under elementary equivalence relation between $\mathcal{L}$-structures, as well as under ultraproducts of those structures.

As in the second order case, it is easy to show that ${\rm DCI}$ is equivalent to its bounded version on intervals.
Also, one can  observe that ${\rm DCI}$ is indeed equivalent to the definable version of the Dedekind completeness in densely ordered abelian groups.
We say that a densely
ordered abelian group $\langle G,+,<,0 \rangle $ is definably
(Dedekind) complete if it has no any definable gap.

\begin{proposition}\label{dcg-dci}
A densely ordered abelian group $G$ is definably complete if and only if it satisfies ${\rm DCI}$.

\end{proposition}
\begin{proof} Assume that $G$ is definably complete and $\phi(v,\bar{a})$
is a formula, with parameters $\bar{a}$, for which
$G$ satisfies the induction hypotheses. If $G \nvDash \forall v \phi(v,\bar{a})$,
then the formula $\psi (v, \bar{a}) =  \forall u<v \phi (u,\bar{a})$ defines a bounded from above nonempty
subset, say $S$, of $G$. By definable completeness, the subset $S$ has a least upper bound in
$G$. This contradicts the second induction hypothesis about $\phi
(v,\bar{a})$.

Conversely, suppose  that $G$ is an ordered abelian group that
satisfies $DCI$. Let $C$ be a definable cut in $G$ which is defined by an  $\mathcal{L}$-formula $\phi(v,\bar{a})$.
 If $C$ does not contain a least upper
bound in $G$, then $G$ satisfies the  hypotheses of  $DCI$ for
$\phi (v,\bar{a})$. Then,  $G\vDash \forall v \phi (v,\bar{a})$ and therefor  $C=G$. Hence, $G$ does not have any definable gap.
\end{proof}
 By \cite{Mi}, definably complete ordered groups are divisible. So, we have the following.
 \begin{cor}
 If $G\vDash DCI$, then $G$ is divisible.
 \end{cor}
\begin{remark} \label{Re}
According to \cite{Dr} a linearly ordered structure $M$ is said to
be \textsl{o}-minimal if any  definable subset of
$M$ is a finite union  of intervals and points in $M$. Any
\textsl{o}-minimal structure is definably complete, but the
converse is not true for arbitrary ordered structures. (See \cite{PS},
Proposition 1.2. and the following comments.) On the other hand, the theory of divisible ordered abelian groups in the language of ordered groups   has quantifier elimination property (See \cite{M}). So, every definable subset of a divisible ordered abelian group in the language
of ordered groups is a finite union of intervals and points.  Hence, those groups are
 \textsl{o}-minimal. The converse was proved
in op. cit. (Theorem 2.1). So, every divisible ordered abelian group is
definably complete and so satisfies ${\rm DCI}$. Note that there are models of  ${\rm DCI}$ that are not  not \textsl{o}-minimal. Here, we give an example. Let $\mathcal{F}$ be a non-principal ultrafilter on the set of natural numbers, $\Bbb{N}$. Let $\mathcal{G}$ be the ultra-product of the structures  $\mathbb{Q}_n=(\mathbb{Q},+,0,<,\{0, \dots , n\})$ with respect to $\mathcal{F}$. Note that for each $n$,   $\mathbb{Q}_n$ satisfies ${\rm DCI}$. So, $\mathcal{G}$ satisfies ${\rm DCI}$ too. But $\mathcal{G}$ is not \textsl{o}-minimal, as the ultra-product of the sets $\{0, \dots , n\})$ is an infinite definable discrete set in  $\mathcal{G}$.

\end{remark}

{\it Definable compactness} in ordered structures  was introduced by \cite{PeS} where it was proved that this concept is equivalent to closed and bounded for definable sets. Also, \cite{EM} contains a weak version this notion which is far from finiteness.  Here, we use the concept of  {\it pseudo-finite set } which was introduced in \cite{F},  to give an alternative definition for definable compactness in  ordered abelian groups. Then, we provide  a definable version of the Heine-Borel
theorem in definably  complete ordered groups using ${\rm DCI}$.

Let $G$ be an $\mathcal{L}$-structure, where $\mathcal{L}$ is an expansion of the language of ordered groups. Then, a definable subset $D\subseteq G$ is said to be {\it pseudo-finite} if it is discrete, closed, and bounded. Note that these are the only first order properties in the language $\mathcal{L}$  describing  the finiteness for sets definable in $\mathcal{L}$-structures. It is easy to see that the union of two pseudo-finite sets is pseudo-finite. Even, the union of a pseudo-finite family of pseudo-finite sets is pseudo-finite (see \cite{F}). In \textsl{o}-minimal structures, a definable set is pseudo-finite if and only if it is finite.

Let $\varphi(x,y)$ be an $\mathcal{L}$-formula. By $\varphi(a,G)$, where $a\in G$,  we mean a set that is defined in $G$ by the formula $\varphi(a,y)$:  $\{b\in G| G\vDash \varphi(a,b)\}$. Then, $\{\varphi(a,G)\}_{a\in G}$ is said to be  a definable family of subsets of $G$.  A definable open cover for a set $X\subseteq G$ is a definable family  $\{\varphi(a,G)\}_{a\in G}$ which covers $X$ and for each $a\in G$, $\varphi(a,G)$ is an open subset of $G$. In the following definition we fix a directed definable family  $\{\psi(t,G)\}_{t\in G}$ of pseudo-finite sets which covers $G$, that is $\psi(t_1,G) \subseteq \psi(t_2,G)$ for $t_1<t_2$, and $\bigcup_{t\in G} \psi(t,G)=G$.
\begin{definition}
Let $G$ be an ordered abelian group. Then,  a definable  subset $D\subseteq G$ is said to be definably  compact if for  every
definable open cover $\{\varphi(a,G)\}_{a\in G}$ of $D$,  there exists  a pseudo-finite sub-cover  $\{\varphi(a,G)\}_{a\in P}$ for some pseudo-finite set $P=\psi(t_0,G)$.
\end{definition}
\begin{theorem}
  All closed and
bounded intervals of definably complete ordered groups are
definably compact.
\end{theorem}

\begin{proof} Let $\langle G,+,<,0 \rangle $ be a definably complete ordered abelian group, and
$\varphi (x,y)$ define an open cover for a closed and bounded interval $[a,b]$ in $G$.
 Let
$$S=\{ x\in G | a<x\leq  b  \rightarrow   \exists t\in G,  \ [a,x]\subseteq \bigcup_{u\in \psi(t,G)} \varphi
(u,G)  \}.$$

Let $\theta (x,a,b)$ be an $\mathcal{L}$-formula which defines  the set $S$ in the $\mathcal{L}$-structure $G$.
We show that $\theta (x,a,b)$ satisfies the conditions of the ${\rm DCI}_{\theta}$.

There exists $x_0\in G$ such that $a\in \varphi(x_0,G)$. Since $\varphi(x_0,G)$ is open, then there is $c\in (a,b)$ such that $[a,c)\subseteq \varphi(x_0,G)$.
On the other hand the family $\psi$ covers $G$, hence there is an element $t_0\in G$ such that $x_0\in \psi(t_0,G)$. Thus, $G$ satisfies $\theta (g,a,b)$ for every $g\in (-\infty , c)$.

Now, let $h\in (a,b) $ and $G$ satisfy $\forall g<h\theta (g,a,b)$. There are $x_1, h_1, h_2 \in G$ such that $a<h_1<h<h_2<b$ and $(h_1, h_2) \subset \varphi (x_1, G)$. There is an element $t_1\in G$ with $x_1\in \psi(t_1, G)$. By the assumption, $G$ satisfies  $\theta (h_1,a,b)$. So, there is a pseudo finite set $\psi(t_2, G)$ such that $[a,h_1]\subseteq \bigcup_{u\in \psi(t_2,G)} \varphi(u,G)$. Since the family $\psi$ is directed, then there is $t_3\in G$ such that $\psi(t_1, G), \psi(t_2, G) \subseteq \psi(t_3, G)$. Hence, $[a,h_2]\subseteq \bigcup_{u\in \psi(t_3,G)} \varphi(u,G)$.

Therefor the conditions of ${\rm DCI}_{\theta}$ hold in $G$. By Proposition \ref{dcg-dci}, $G$ satisfies ${\rm DCI}_{\theta}$. Hence, $S=G$, and so there is a pseudo finite set $\psi(t, G)$ for some $t\in G$ such that  $[a,b]\subseteq \bigcup_{u\in \psi(t,G)} \varphi(u,G)$.

\end{proof}

\end{document}